\documentclass[12pt,leqno]{amsart}
\usepackage{amssymb,amsthm,amsmath,latexsym}

\newtheorem{theorem}{\sc Theorem}[section]

\newtheorem{lem}[theorem]{\sc Lemma}

\newtheorem{cor}[theorem]{\sc Corollary}

\newtheorem*{thmA}{Theorem A}
\newtheorem*{thmB}{Theorem B}

\title[Bounded Engel elements]{Bounded Engel elements in residually finite groups}
\author[Bastos]{Raimundo Bastos}
\address{(Bastos) Departamento de Matem\'atica, Universidade de Bras\'ilia,
Brasilia-DF, 70910-900 Brazil}
\email{bastos@mat.unb.br}
\author[Silveira]{Danilo Silveira}
\address{(Silveira) Departamento de Matem\'atica, Universidade Federal de Goi\'as,
Catal\~ao-GO, 75704-020 Brazil}
\email{sancaodanilo@gmail.com}
\thanks{The first author was partially supported by FAPDF/Brazil.}

\begin{document}

\maketitle

\begin{abstract}
Let $q$ be a prime. Let $G$ be a residually finite group satisfying an identity. Suppose that for every $x \in G$ there exists a $q$-power $m=m(x)$ such that the element $x^m$ is a bounded Engel element. We prove that $G$ is locally virtually nilpotent. Further, let $d,n$ be positive integers and $w$ a non-commutator word. Assume that $G$ is a $d$-generator residually finite group in which all $w$-values are $n$-Engel. We show that the verbal subgroup $w(G)$ has $\{d,n,w\}$-bounded nilpotency class.
\end{abstract}

\maketitle

\section{Introduction}

Given a group $G$, an element $g\in G$ is called a (left) Engel element if for any $x\in G$ there exists a positive integer $n=n(x,g)$ such that $[x,_n g]=1$, where the commutator $[x,_n g]$ is defined inductively by the rules $$[x,_1 g]=[x,g]=x^{-1}g^{-1}xg\quad {\rm and,\, for}\; n\geq 2,\quad [x,_n g]=[[x,_{n-1} g],g].$$
If $n$ can be chosen independently of $x$, then $g$ is called a (left) $n$-Engel element, or more generally a bounded (left) Engel element. The group $G$ is an Engel group (resp. an $n$-Engel group) if all its elements are Engel (resp. $n$-Engel).

A celebrated result due to Zelmanov \cite{ze1,ze2,ze16} refers to the positive solution of the Restricted Burnside Problem (RBP for short): every residually finite group of bounded exponent is locally finite. The group $G$ is said to have a certain property locally if any finitely generated subgroup of $G$ possesses that property. An interesting result in this context, due to Wilson \cite{w}, states that every $n$-Engel residually finite group is locally nilpotent. Another result that was deduced following the positive solution of the RBP is that given positive integers $m,n$, if $G$ is a residually finite group in which for every $x \in G$ there exists a positive integer $q=q(x) \leqslant m$ such that $x^q$ is $n$-Engel, then $G$ is locally virtually nilpotent \cite{Ba}. We recall that a group possesses a certain property virtually if it has a subgroup of finite index with that property. For more details concerning Engel elements in residually finite groups see \cite{Ba,BMTT,BSTT,STT-var,STT}. 

One of the goals of the present article is to study residually finite groups in which some powers are bounded Engel elements. We establish the following result.

\begin{thmA}\label{main}
Let $q$ be a prime. Let $G$ be a residually finite group satisfying an identity. Suppose that for every $x \in G$ there exists a $q$-power $m=m(x)$ such that the element $x^m$ is a bounded Engel element. Then $G$ is locally virtually nilpotent.
\end{thmA}

A natural question arising in the context of the above theorem is whether the theorem remains valid with $m$ allowed to be an arbitrary natural number rather than $q$-power. This is related to the conjecture that if $G$ is a residually finite  periodic group satisfying an identity, then the group $G$ is locally finite (Zelmanov, \cite[p. 400]{zelm}). Note that the hypothesis that $G$ satisfies an identity is really needed. For instance, it is well known that there are residually finite $p$-groups that are not locally finite (Golod, \cite{G}). In particular, these groups cannot be locally virtually nilpotent. Similar examples have been obtained independently by Grigorchuk, Gupta-Sidki and Sushchansky and are published in \cite{Gri,GS,S}, respectively.   

Recall that a group-word $w=w(x_1,\dots,x_s)$ is a nontrivial element of the free group $F = F(x_1,\dots,x_s)$ on free generators $x_1,\dots,x_s$. A word is a commutator word if it belongs to the commutator subgroup $F'$. A non-commutator word $u$ is a group-word such that the sum of the exponents of some variable involved in it is non-zero. A group-word $w$ can be viewed as a function defined in any group $G$. The subgroup of $G$ generated by the $w$-values is called the verbal subgroup of $G$ corresponding to the word $w$. It is usually denoted by $w(G)$. However, if $k$ is a positive integer and $w = x_1^k$, it is customary to write $G^{k}$ rather than $w(G)$.

There is a well-known quantitative version of Wilson's theorem, that is, if $G$ is a $d$-generator residually finite $n$-Engel group, then $G$ has $\{d,n\}$-bounded nilpotency class. As usual, the expression ``$\{a,b,...\}$-bounded'' means ``bounded from above by some function which depends only on parameters $a,b,...$''. We establish the following related result.   

\begin{thmB}\label{main.2}
Let $d,n$ be positive integers and $w$ a non-commutator word. Assume that $G$ is a $d$-generator residually finite group in which all $w$-values are $n$-Engel. Then the verbal subgroup  $w(G)$ has $\{d,n,w\}$-bounded nilpotency class.
\end{thmB}

A non-quantitative version of the above theorem already exists in the literature. It was obtained in \cite[Theorem C]{BSTT}.  

The paper is organized as follows. In the next section we describe some important ingredients of what are often called ``Lie methods in group theory''. Theorems A and B are proved in Sections 3 and 4, respectively. The proofs of the main results rely of  Zelmanov's techniques that led to the solution of the RBP \cite{ze1,ze2,ze16}, Lazard's criterion for a pro-$p$ group to be $p$-adic analytic \cite{la}, and a result of Nikolov and Segal \cite{NS} on verbal width in groups. 

\section{Associated Lie algebras}

Let $L$ be a Lie algebra over a field $\mathbb{K}$. We use the left normed notation: thus if $l_1,l_2,\dots,l_n$ are elements of $L$, then $$[l_1,l_2,\dots,l_n]=[\dots[[l_1,l_2],l_3],\dots,l_n].$$ We recall that an element $a\in L$ is called {\it ad-nilpotent} if there exists a positive integer $n$ such that $[x,{}_na]=0$ for all $x\in L$. When $n$ is the least integer with the above property then we say that $a$ is ad-nilpotent of index $n$. 

Let $X\subseteq L$ be any subset of $L$. By a commutator of elements in $X$, we mean any element of $L$ that can be obtained from
elements of $X$ by means of repeated operation of commutation with an arbitrary system of brackets
including the elements of $X$. Denote by $F$ the free Lie algebra over $\mathbb{K}$ on countably many free
generators $x_1,x_2,\dots$. Let $f=f(x_1,x_2,\dots,x_n)$ be a non-zero element of $F$. The algebra $L$ is said to satisfy the identity $f\equiv 0$ if
$f(l_1,l_2,\dots,l_n)=0$ for any $l_1,l_2,\dots,l_n\in L$. In this case we say that $L$ is PI. Now, we recall an important theorem of Zelmanov 
\cite[Theorem 3]{zelm} that has many applications in group theory.
\begin {theorem}\label{1} 
Let $L$ be a Lie algebra over a field generated by a finite set. Assume that $L$ is PI and that each commutator in the generators is ad-nilpotent. Then $L$ is nilpotent.
\end{theorem}

\subsection{On Lie Algebras Associated with Groups}

Let $G$ be a group and $p$ a prime. Let us denote by $D_i=D_i(G)$ 
the $i$-th dimension subgroup of $G$ in characteristic
$p$. These subgroups form a central series of $G$
known as the {\it Zassenhaus-Jennings-Lazard series} (see \cite[p. 250]{Huppert2} for more details). Set $L(G)=\bigoplus D_i/D_{i+1}$. 
Then $L(G)$ can naturally be viewed as a Lie algebra 
over the field ${\mathbb F}_p$ with $p$ elements. 

The subalgebra of $L(G)$ generated by $D_1/D_2$ will be denoted by $L_p(G)$. The nilpotency of $L_p(G)$ has strong influence in the structure of a finitely generated group $G$. According to Lazard \cite{l2} the nilpotency of $L_p(G)$ is equivalent to $G$ being $p$-adic analytic (for details see  \cite[A.1 in Appendice and  Sections 3.1 and 3.4 in Ch.\ III]{l2} or  \cite[1.(k) and 1.(o) in Interlude A]{GA}).

\begin{theorem}\label{3} 
Let $G$ be a finitely generated pro-$p$ group. If $L_p(G)$ is nilpotent, then $G$ is $p$-adic analytic.
\end{theorem}

Let $x\in G$ and let $i=i(x)$ be the largest positive integer such that $x\in D_i$ (here, $ D_i $ is a term of the $ p $-dimensional central series to $ G $). We denote by $\tilde{x}$ the element $xD_{i+1}\in L(G)$.  We now quote two results providing sufficient conditions for $\tilde{x}$ to be ad-nilpotent. The first lemma was established in \cite[p. 131]{la}.

\begin{lem}\label{lazard-ad} 
For any $x\in G$ we have $(ad\,{\tilde x})^p=ad\,(\widetilde {x^p})$. Consequently, if $x$ is of finite order $t$ then $\tilde{x}$ is ad-nilpotent of index  at most  $t$.  
\end{lem}

\begin{cor}\label{lemma-lazard}
Let $x$ be an element of a group $G$ for which there exists a positive integer $m$ such that $x^m$ is $ n $-Engel. Then $\tilde{x}$ is ad-nilpotent.
\end{cor}

The following result was established by Wilson and Zelmanov in \cite{wize}.

\begin{lem}\label{identity} Let $G$ be a group satisfying an identity. Then for 
each prime number $p$ the Lie algebra $L_p(G)$ is PI. 
\end{lem}

\section{Proof of Theorem A}

Recall that a group is locally graded if every nontrivial finitely generated subgroup has a proper subgroup of finite index. Interesting classes of groups (e.g., locally finite groups, locally nilpotent groups, residually finite groups) are locally graded (see \cite{LMS, M} for more details).     

It is easy to see that a quotient of a locally graded group need not be locally graded (see for instance \cite[6.19]{Rob}). However, the next result gives a sufficient condition for a quotient to be locally graded \cite{LMS}.

\begin{lem} \label{HP}
Let $G$ be a locally graded group and $N$ a normal locally nilpotent subgroup of $G$. Then $G/N$ is locally graded.
\end{lem}

In \cite{zelm}, Zelmanov has shown that if $G$ is a residually finite $p$-group which satisfies a nontrivial identity, then $G$ is locally finite. Next, we extend this result to the class of locally graded groups.

\begin{lem} \label{lem.graded} Let $p$ be a prime. Let $G$ be a locally graded $p$-group which satisfies an identity. Then $G$ is locally finite.
\end{lem}

\begin{proof}
Choose arbitrarily a finitely generated subgroup $H$ of $G$. Let $R$ be the finite residual of $H$, i.e., the intersection of all subgroups of finite index in $H$. If $R=1$, then $H$ is a finitely generated residually finite group. By Zelmanov's result \cite[Theorem 4]{zelm}, $H$ is finite. So it suffices to show that $H$ is residually finite. We argue by contradiction and suppose that $R \neq 1$. By the above argument, $H/R$ is finite and thus $R$ is finitely generated. As $R$ is locally graded we have that $R$ contains a proper subgroup of finite index in $H$, which gives a contradiction. Since $H$ be chosen arbitrarily, we now conclude that $G$ is locally finite, as well. The proof is complete. 
\end{proof}

We denote by $\mathcal{N}$ the class of all finite nilpotent groups. The following result is a straightforward corollary of \cite[Lemma 2.1]{w} (see  \cite[Lemma 3.5]{P-2000} for details).

\begin{lem}\label{Wilson}
Let $G$ be a finitely generated residually-$\mathcal{N}$ group. For each prime $p$, let $R_p$ be the intersection of all normal subgroups of $G$ of finite $p$-power index. If $G/R_p$ is nilpotent for each prime $p$, then $G$ is nilpotent.
\end{lem}

We are now in a position to prove Theorem A. 

\begin{proof}[Proof of Theorem A] 
Recall that $G$ is a residually finite group satisfying an identity in which for every $x \in G$ there exists a $q$-power $m=m(x)$ such that the element $x^m$ is a bounded Engel element. We need to prove that every finitely generated subgroup of $G$ is virtually nilpotent

Firstly, we prove that all bounded Engel elements (in $G$) are contained in the Hirsch-Plotkin radical of $G$. Let $H$ be a subgroup generated by finitely many bounded Engel elements in $G$, say $H = \langle h_1, \ldots, h_t \rangle$, where $h_i$ is a bounded Engel element in $G$ for every $i=1,\ldots,t$. Since finite groups generated by Engel elements are nilpotent \cite[12.3.7]{Rob}, we can conclude that $H$ is residually-$\mathcal{N}$. As a consequence of Lemma \ref{Wilson}, we can assume that $G$ is residually-(finite $p$-group) for some prime $p$. Let $L=L_p(H)$ be the Lie algebra associated with the Zassenhaus-Jennings-Lazard series $$H=D_1\geq D_2\geq \cdots$$ of $H$. Then $L$ is generated by $\tilde{h}_i=h_i D_2$, $i=1,2,\dots,t$. Let $\tilde{h}$ be any Lie-commutator in $\tilde{h}_i$ and $h$ be the group-commutator in $h_i$ having the same system of brackets as $\tilde{h}$. Since for any group commutator $h$ in $h_1\dots,h_t$ there is a $q$-power $m=m(h)$ and a positive integer $n=n(h)$ such that $h^m$ is $n$-Engel, Corollary \ref{lemma-lazard} shows that any Lie commutator in $\tilde h_1\dots,\tilde h_t$ is ad-nilpotent. On the other hand, $H$ satisfies an identity and therefore, by Lemma \ref{identity}, $L$ satisfies some non-trivial polynomial identity. According to Theorem \ref{1} $L$ is nilpotent. Let $\hat{H}$ denote the pro-$p$ completion of $H$. Then $L_p(\hat{H})=L$ is nilpotent and $\hat{H}$ is a $p$-adic analytic group by Theorem \ref{3}. By  \cite[1.(n) and 1.(o) in Interlude A]{GA}), $\hat{H}$ is linear, and so therefore is $H$. Clearly $H$ cannot have a free subgroup of rank 2 and so, by Tits' Alternative \cite{tits}, $H$ is virtually soluble. By \cite[12.3.7]{Rob}, $H$ is soluble. Since $h_1,\dots,h_t$ have been chosen arbitrarily, we now conclude that all bounded Engel elements are in the Hirsch-Plotkin radical of $G$. 

Let $H$ be a finitely generated subgroup of $G$, and $K$ be the subgroup generated by all bounded Engel elements (in $G$) contained in $H$. Now, we need to prove that $K$ is a nilpotent subgroup of finite index in $H$. By the previous paragraph, $K$ is locally nilpotent. By Lemma \ref{HP}, $H/K$ is a locally graded $q$-group. Since $G$ satisfies a nontrivial identity, by Lemma \ref{lem.graded}, $H/K$ is finite and so, $K$ is finitely generated. From this we deduce that $K$ is nilpotent. The proof is complete. 
\end{proof}

\section{Proof of Theorem B}

Combining the positive solution of the RBP with the result \cite[Theorem C]{BSTT} one can show that if $u$ is a non-commutator word and $G$ is a finitely generated residually finite group in which all $u$-values are $n$-Engel, then the verbal subgroup $u(G)$ is nilpotent. This section is devoted to obtain a quantitative version of the aforementioned result.   

The proof of Theorem B require the following lemmas. 

\begin{lem} \label{prop.power}
Let $d,m,n$ positive integers. Let $G$ be a $d$-generator residually finite group in which $x^m$ is $n$-Engel for every $x\in G $. Then the subgroup $G^m$ has $ \{d,m,n\} $-bounded nilpotency class. 
\end{lem}

\begin{proof}
Let $H=G^m$. By \cite[Theorem C]{BSTT}, $H$ is locally nilpotent. Moreover, Lemma \ref{HP} ensures us that the quotient group $G/H$ is locally graded. By Zelmanov's solution of the RBP, locally graded groups of finite exponent are locally finite (see for example \cite[Theorem 1]{M}), and so $G/H$ is finite of $\{d,m\}$-bounded order. We can deduce from \cite[Theorem 6.1.8(ii)]{Rob} that $H$ has $\{d,m\}$-boundedly many generators. In particular, $H$ is nilpotent. In order to complete the proof, we need to show that $H$ has $\{d,m,n\}$-bounded class yet. 

Note that there exists a family of normal and finite index subgroups  $\{N_i\}_{i\in \mathcal{I}}$ in $G$ which are all contained in $H$ such that $H$ is isomorphic to a subgroup of the Cartesian product of the finite quotients $H/N_i$. We show that all quotients have $\{d,m,n\}$-bounded class. Indeed, we have $H/N_i=(G/N_i)^m$. Note that $H$ is $\{d,m\}$-boundedly generated. Thus, by \cite[Theorem 1]{NS}, $H/N_i$ is $\{d,m\}$-boundedly generated where any generator is an $m$-th power which is an $n$-Engel element. By \cite[Lemma 2.2]{STT-var}, there exists a number $c$ depending only on $\{d,m,n\}$ such that each factor $H/N_i$ has nilpotency class at most $c$. So $H$ is of nipotency class at most $c$, as well. The proof is complete. 
\end{proof}

A well known theorem of Gruenberg says that a soluble group generated by finitely many Engel elements is nilpotent (see \cite[12.3.3]{Rob}). We will require a quantitative version of this theorem whose proof can be found in \cite[Lemma 4.1]{shusa}.
\begin{lem}\label{gru}
Let $G$ be a group generated by $m$ elements which are $n$-Engel and suppose that $G$ is soluble with derived length $d$. Then $G$ is nilpotent of $\{d,m,n\}$-bounded class. 
\end{lem}

For the reader's convenience we restate Theorem B.

\begin{thmB}Let $d,n$ be positive integers and $w$ a non-commutator word. Assume that $G$ is a $d$-generator residually finite group in which all $w$-values are $n$-Engel. Then the verbal subgroup  $w(G)$ has $\{d,n,w\}$-bounded nilpotency class.
\end{thmB}

\begin{proof}
Let $w = w(x_1,...,x_r)$ be a non-commutator word. We may assume that the sum of the exponents of $x_1$ is $k \neq 0$. Substitute $1$ for $x_2,...,x_r$ and an arbitrary element $g \in G$ for $x_1$. We see that $g^k$ is a $w$-value for every $g \in G$. Thus every $k$-th power is $n$-Engel in $G$. Lemma \ref{prop.power} ensures that $G^k$ has $\{d,n,w,\}$-bounded nilpotency class.

Following an argument similar to that used in the proof of Lemma \ref{prop.power} we can deduce that the verbal subgroup $w(G)$ is nilpotent.  By Zelmanov's solution of the RBP, locally graded groups of finite exponent are locally finite (see for example \cite[Theorem 1]{M}), and so $G/G^k$ is finite of $\{d,w\}$-bounded order. Thus, the verbal subgroup $w(G)$ has $\{d,m,w\}$-bounded derived length. 

Note that there exists a family of normal and finite index subgroups  $\{N_i\}_{i\in \mathcal{I}}$ in $G$ that are all contained in $w(G)$ such that $w(G)$ is isomorphic to a subgroup of the Cartesian product of the finite quotients $w(G)/N_i$. We show that all quotients $w(G)/N_i$ have $\{d,n,w\}$-bounded class. Indeed, we have $w(G)/N_i=w(G/N_i)$. We also have $w(G)$ is $\{d,w\}$-boundedly generated. By \cite[Theorem 3]{NS} each quotient $w(G)/N_i$ is $\{d,w\}$-boundedly generated by $w$-values which are $n$-Engel elements. Since $w(G)$ has $\{d,m,w\}$-bounded derived length, according to Lemma \ref{gru} we can deduce that $w(G)/N_i$ has  $\{d,n,w\}$-bounded nilpotency class Thus, $w(G)$ has  $\{d,n,w\}$-bounded nilpotency class, as well. This completes the proof.
\end{proof}

\end{document}